\documentclass{amsart}

\usepackage{amsmath,amssymb, amsfonts}
\usepackage{hyperref}
\usepackage[latin1,utf8]{inputenc}
\usepackage[T1]{fontenc}

\newcommand{\um}{^^c3^^a4}
\newcommand{\LieDerivative}{\mathcal L_\xi}

\newcommand{\manifold}{\mathcal }

\newtheorem{theorem}{Theorem}
\theoremstyle{plain}

\newtheorem{corollary}{Corollary}

\newtheorem{example}{Example}

\newtheorem{proposition}{Proposition}
\newtheorem{remark}{Remark}

\numberwithin{equation}{section}

\title{Classification results for three-dimensional (para)contact metric and
 almost (para)cosymplectic $(\kappa,\mu)$-spaces.}
\author{Piotr Dacko}
\begin{document}
\begin{abstract}
 It is provided an overview of existed results concerning classification of 
contact metric, almost cosymplectic and almost Kenmotsu $(\kappa,\mu)$-manifolds.
In the case of dimension three  it is described in full details structure 
 of contact metric or almost cosymplectic $(\kappa,\mu)$-spaces. The second part of the 
paper addresses three-dimensional paracontact metric and almost paracosymplectic $(\kappa,\mu)$-spaces. 
There is obtained local classification of paracontact metric $(\kappa,\mu)$-spaces,  and
 almost paracosymplectic $(\kappa,\mu)$-spaces, for every possible value of $\kappa$..  
\end{abstract}
\maketitle

\section{Introduction}   

In this paper there are studied almost contact metric 
and almost paracontact metric three-dimensional $(\kappa,\mu)$-manifolds. 
In particular  contact metric manifold or almost cosymplectic 
 $(\kappa,\mu)$-space can be realized as three-dimensional unimodular Lie group equipped with left-invariant 
almost contact metric structure. As these manifolds are already classified the paper is an overview of 
existing result in case of dimension three. Novelty of our approach is, that there is created on $\mathbb R^3$, 
smooth one-parameter family of almost contact metric structures, such that for particular
value of parameter $\mathbb R^3$ turns into contact metric $(\kappa,\mu)$-space of 
almost cosymplectic $(\kappa,\mu)$-space. From other hand  every such 
manifold is locally isometric to unique simply connected connected Lie group equipped with left-invariant 
almost contact metric structure.   

The case of paracontact metric or almost paracosymplectic $(\kappa,\mu)$-spaces 
is more challenging. Still up to the author knowledge there is no 
full classification of such manifolds. For the literature on the subject see \cite{CPM}, \cite{CPM:DiTer}, \cite{CPM:Er:Mur}.

The reason to provide detailed study of three-dimensional manifolds is connected to 
possible decomposition theorem. Idea is that every higher dimensional $(\kappa,\mu)$-spaces,
is constructed from three dimensional $(\kappa,\mu)$-space. Conjecturing there is reverse 
decomposition classification of three-dimensional $(\kappa,\mu)$-spaces is enough to provide 
local classification in all dimensions.

As almost contact metric and almost paracontact metric manifolds share some 
similarities it seems to be interesting to mix  these classes. In the sense to consider 
pseudo-Riemannian manifolds with $\phi$-4 structure: 
$\phi^4 = Id -\eta\otimes \xi$.  Such manifold is equipped with 
corresponding fundamental form and usual classes can be defined:
contact metric with pseudo-metric,  almost cosymplectic with 
pseudo-metric, etc.   For example we can equip odd-dimensional 
Lorentzian manifold with structure of contact metric manifold with 
Lorentzian metric.

\section{Preliminaries}
All manifolds considered in the paper are smooth, connected without boundary. Tensor 
fields on smooth manifold are assumed to be smooth. If not otherwise stated letters 
$X$, $Y$, $Z$, ... denote vector fields.

\subsection{Almost contact metric manifolds}
  Let $\mathcal M$ be $(2n+1)$-dimensional manifold,  
  $n \geqslant 1$.
  Almost contact metric structure is a quadruple of tensor fields 
  $(\phi, \xi, \eta, g)$, where $\phi$ is $(1,1)$-tensor field, 
  $\xi$ is a characteristic (or Reeb) vector field, 
  $\eta$ a chracteristic 1-form and $g$ - a Riemannian metric. By definition 
  \begin{align}
  & \phi^2 = -Id +\eta\otimes \xi, \quad \eta(\xi) =1, \\
 & g(\phi X, \phi Y) = g(X,Y)- \eta(X)\eta(Y). 
  \end{align} 
     Tensor field $\Phi(X,Y)= g (X,\phi Y)$ is 
  skew-symmetric, $\Phi(X,Y)+\Phi(Y,X)=0$. It determines a 2-form on 
  $\mathcal M$. 
   called fundamental form. From definition of $\Phi$, there is
  $\eta\wedge \Phi^n \neq 0$,
  at every point of $\mathcal M$. 
  Manifold equipped 
  with some fixed 
  almost contact metric structure is called almost contact metric 
  manifold. 
  
  Denote by $N_S$ Nijenhuis torsion of a $(1,1)$-tensor field $S$
\begin{equation*}
N_S(X,Y)= S^2[X,Y]+[SX,SY]-S([SX,Y]+[X,SY]).
\end{equation*} 
 Almost contact metric manifold $\mathcal M$
  is called normal if $N_\phi +2\,d\eta \otimes \xi =0$. Normality is equivalent to the 
  existence of complex structure on the product 
$\widetilde{ \mathcal M} = \mathcal M \times \mathbb \mathbb S^1$, 
with the circle. 

Almost contact metric  manifold $\mathcal M$ is 
  called contact metric if $d\eta = \Phi$, almost cosymplectic (or 
  almost coK\um hler) if $d\eta=0$, $d\Phi=0$,  almost Kenmotsu 
  if $d\eta=0$, $d\Phi = 2\eta\wedge\Phi$. Assuming normality
  we obtain Sasakian (contact metric and normal), cosymplectic 
  (or coK\um hler) and Kenmotsu manifolds. Almost 
  contact metric manifold with maximal isometry group 
  is locally isometric either to Sasakian manifold 
  of constant sectional curvature $c=+1$, 
  Kenmotsu manifold of constant sectional curvature $c=-1$, or 
  cosymplectic manifold of constant sectional curvature $c=0$.  
  General literature on the subject can be found eg. in 
  \cite{Blair},  \cite{CPM:DeNi:Yud}, \cite{CG2}, \cite{Dileo}, 
  \cite{GY}, \cite{LL}, \cite{O1}. 
  
Let $\nabla$ denote Levi-Civita connection of the metric, 
    $R(X,Y) = [\nabla_X,\nabla_Y]Z - \nabla_{[X,Y]}Z$, 
  curvature of $\nabla$.
     Define $h=\frac{1}{2}\mathcal L_\xi\phi$. Let $\kappa$, 
     $\mu$,  be real constants. Almost contact metric manifold  
     is called 
     $(\kappa,\mu)$-space if  
\begin{align}
     \label{km:nul:c}
 &    R(X,Y)\xi  =  \kappa(\eta(Y)X-\eta(X)Y) +\mu (\eta(Y)hX -\eta(X)hY).
 \end{align}
 Note if $h=0$ 
 it is not possible to determine $\mu$. But condition is still formally 
 valid for any possible value of  $\mu$.  Ambiguity also arrives if 
 \begin{equation}
 \label{k:nul:c}
R(X,Y)\xi = \kappa (\eta(Y)X-\eta(X)Y),
\end{equation}
for example in case $\mu =0$. 

We denote $\mathcal D = \{ \eta = 0 \}$, kernel distribution of characteristic form. $\mathcal D$-homothety  
is deformation of almost contact metric structure, $\alpha \in\mathbb R^+$ 
\begin{align*}
& \phi \mapsto \phi'=\phi, \quad \xi \mapsto \xi'=\alpha^{-1}\xi, \quad 
\eta \mapsto \eta' = \alpha \eta, \\
& g \mapsto g' = \alpha g+ \alpha(\alpha-1)\eta\otimes \eta.
\end{align*}

 \begin{theorem}[Blair, Koufogiorgos, Papantoniou, 1995, \cite{BKP} ] Let $\mathcal M$ be contact metric 
 $(\kappa, \mu)$-manifold. Then $\kappa \leqslant 1$. The following relations 
 hold
 \begin{align}
\label{cm:nab:fi}
& (\nabla_X\phi)Y = g(X,Y+hY)\xi -\eta(Y)(X+hX),  \\
\label{cm:nab:h}
& (\nabla_Xh)Y = ((1-\kappa)g(X,\phi Y)-g(X,\phi h Y))\xi -\\ 
& \qquad  \eta(Y)((1-\kappa)\phi X + \phi h X) - \mu\eta(X)\phi h Y.\nonumber 
\end{align}
 \end{theorem}
For contact metric non-Sasakian $(\kappa,\mu)$-space $\mathcal M$, let define 
\begin{equation}
I_{\mathcal M} = (1-\mu/2)/\sqrt{1-\kappa},
\end{equation}
$I_{\mathcal{M}}$ is called Boeckx invariant. 
 \begin{theorem}[Boeckx, 2000, \cite{Boe}] Let $\mathcal M_i$, $i=1,2$,  be 
two non-Sasakian $(\kappa_i,\mu_i)$-spaces of the same dimension. Then 
$I_{\mathcal{M}_1}=I_{\mathcal{M}_2}$ if and only if, up to $\mathcal{D}$-homothetic 
transformation, the two spaces are locally isometric as contact metric spaces.
In particular, if both spaces are simply connected and complete, they are globally 
isometric up to $\mathcal{D}$-homothetic transformation.  
\end{theorem}
 
 The cited result  is base for classification of non-Sasakian contact metric 
  $(\kappa,\mu)$-manifolds. 
 It is enough to provide an example  of  manifold $\mathcal M$, 
 for every allowable value $I$
 of Boeckx invariant, such that $I_{\mathcal M} = I$.

 For almost cosymplectic manifold 
 distribution $\{ \eta = 0 \}$ is  completely integrable. Therefore it determines canonical foliation
of the manifold. 
 Let $\mathcal F$ denote a leaf  passing through some point $\in \mathcal M$. 
 Then $\mathcal F$ inherits structure of almost K\um hler manifold.  Assuming 
 structure is K\um hler for every leaf manifold is called almost cosymplectic 
 with K\um hler leaves.
 
 \begin{theorem}[Olszak, 1987, \cite{O3}] Define $A = -\nabla\xi$.  Almost cosymplectic 
 manifold has K\um hlerian leaves if and only if  
 \begin{equation}
(\nabla_X\phi)Y= -g(\phi AX, Y) + \eta(Y)\phi A  X.
\end{equation}
\end{theorem}

 \begin{theorem}[Dacko, Olszak, 2005, \cite{DO2}]
 Let $\mathcal M$ be non-cosymplectic almost 
cosymplectic $(\kappa,\mu)$-manifold. Then $\kappa \leqslant 0$. 
If $\kappa = 0$, $\mathcal M$ is locally isometric 
to product of  real line and almost K\um hler manifold. For $\kappa <0$,
 $\mathcal M$ has K\um hler leaves and each leaf is locally flat K\um hler manifold. 
There is following identity
\begin{gather}
\label{ac:lfi:lh:lfih}
\mathcal L_\xi \phi  =  2 h, \quad \mathcal L_\xi h = -2\kappa \phi -\mu \phi h, \quad 
\mathcal L_\xi (\phi h ) = \mu h,
\end{gather}
\end{theorem} 
The theorem allows to classify, by analytic solution, 
 almost cosymplectic $(\kappa,\mu)$-manifolds 
in terms of so-called models. For every $\mu\in \mathbb R$ there is  almost cosymplectic 
$(-1,\mu)$-manifold - called model - and every other $(\kappa, \mu)$-manifold 
is locally isometric up to $\mathcal D$-homothety to particular model 
\cite{DO3}. The value $\frac{\mu}{\sqrt{-\kappa}}$, $\kappa < 0$ is 
$\mathcal D$-homothety invariant. We set 
$C_{\mathcal M } =\frac{-\mu/2}{\sqrt{-\kappa}}$, $\kappa < 0$ and 
call $C_{\mathcal M}$ \emph{Dacko-Olszak invariant} of almost
 cosymplectic $(\kappa,\mu)$-manifold. 

For almost Kenmotsu manifolds there are following basic results. 
\begin{theorem}[Dileo, Pastore, 2009] Let $\mathcal M$ be almost Kenmotsu 
$(\kappa,\mu)$-manifold. Then $\kappa  = -1$, $h=0$ and $\mathcal M$ 
is locally warped product  of an almost K\um hler manifold and open interval. 
If $\mathcal M$ is locally symmetric then $\mathcal M$ is locally isometric 
to the hyperbolic space $\mathbb H(-1)$ of constant sectional curvature $-1$.
\end{theorem}
  \begin{theorem}[Dileo, Pastore, 2009 ]
  Let $\mathcal M$ be almost Kenmotsu 
manifold  such that $ h \neq 0$ and
\begin{equation}
\label{ak:km:nul:c}
R(X,Y)\xi =\kappa (\eta(Y)X-\eta(X)Y) + \mu( \eta(Y) h \phi X - \eta(X)h\phi Y), 
\end{equation}
then $\mathcal M$ is locally isomeric to warped products 
\begin{equation}
\mathbb H^{n+1}(\kappa -2\lambda) \times_f \mathbb R^n, \quad
B^{n+1}(\kappa+2\lambda)\times_{f'} \mathbb R^n, 
\end{equation}
where $H^{n+1}(\kappa -2\lambda)$ is the hyperbolic space of constant 
sectional curvature $k-2\lambda < -1$, $B^{n+1}(\kappa+2\lambda)$ 
is a space of constant sectional curvature $\kappa +2\lambda \leqslant 0$,
$f = c e^{(1-\lambda)t}$, $f' = c' e^{(1+\lambda)t}$, 
$\lambda =\sqrt{|1+\kappa|}$.  
\end{theorem}

 It is known that almost Kenmotsu manifold as Riemannian 
manifold is locally conformal to almost cosymplectic manifold. For this point 
of view see \cite{F1},  \cite{O3}. In \cite{PS1} authors study generalized nullity 
distribution on almost Kenmotsu manifold. 

Curvature properties of more general class of almost 
cosymplectic and almost Kenmotsu so-called $(\kappa,\mu,\nu)$-manifolds are studied in 
\cite{CarMarMol}, \cite{DO2}.

\subsection{Almost paracontact metric manifolds} 
 Almost 
paracontact metric structure on $(2n+1)$-dimensional manifold $\manifold M$, $n \geqslant 1$, 
 is a quadruple 
of tensor fields $(\phi, \xi,\eta,g)$, where $\phi$ is an  affinor, 
$\xi$  characteristic vector field
\footnote{We will use also term Reeb vector field}, 1-form $\eta$ and 
pseudo-Riemannian metric $g$ of signature $(n+1,n)$. It is assumed that 
\begin{equation}
\begin{array}{l}
\phi^2 = Id-\eta\otimes \xi,\quad \eta(\xi) =1 , \\[+4pt]
g(\phi X,\phi Y) = - g(X,Y)+\eta(X)\eta(Y).
\end{array}
\end{equation}

The immediate consequences of the definition are that tensor field 
$\phi$ - viewed as linear map on tangent space - has three 
eigenvalues $(0,-1,1)$ of multiplicities $(1,n,n)$. Eigendistributions 
$ p \mapsto \mathcal V_p^{-1} $, $p \mapsto \mathcal V_p^{+1}$ are totally isotropic,
$g(\mathcal V^-, \mathcal V^-) = 0$, $g(\mathcal V^+,\mathcal V^+)=0$.  
Tensor field $\Phi(X,Y) = g(X,\phi Y)$ is 
 fundamental 2-form,
$\eta\wedge \Phi^n \neq 0$ everywhere. Manifold $\manifold M $ equipped with fixed
almost paracontact metric structure is called almost paracontact
metric manifold.

Let $\manifold M$ be an almost paracontact metric manifold. $\mathcal M$ is
said to be: 
\begin{enumerate}
\item
 normal if 
$$
N_\phi -2d\eta\otimes \xi =0;
$$
\item
paracontact metric if 
$$ 
d\eta = \Phi;
$$
\item
almost paracosymplectic if 
$$
d\eta=0, \quad d\Phi =0;
$$
\item
 almost para-Kenmotsu
if 
$$
d\eta=0,\quad d\Phi = 2\eta\wedge \Phi.
$$
\end{enumerate}

Let define $h = \frac{1}{2}\LieDerivative \phi$.  
Applying derivative $\LieDerivative$ to both sides of identity $\phi^2 = Id -\eta\otimes\xi$,  
we obtain  
\begin{equation}
 h \phi +  \phi h  = -\frac{1}{2} (\xi \llcorner d\eta)\otimes \xi,
\end{equation}
therefore $h$ and $\phi$ anti-commute if and only if $\xi \llcorner d\eta =0$.
 For all  mentioned above classes of manifolds - normal etc., the  condition
$\xi \llcorner d\eta = 0$ is satisfied. Therefore for all above classes we have 
$$
h\phi +\phi h=0.
$$

Let denote by $\nabla$
the Levi-Civita connection of $\mathcal M$,
$R_{XY}Z=[\nabla_X,\nabla_Y]Z-\nabla_{[X,Y]}Z $, curvature operator of $\nabla$.
 Almost paracontact metric manifold is called
$(\kappa,\mu)$-space if its curvature satisfies 
\begin{align*}  
 & R_{XY}Z  =  \kappa(\eta(Y)X-\eta(X)Y) +\mu(\eta(Y)hX - \eta(X)hY),  \quad \kappa,\mu \in \mathbb R.
\end{align*}

In similar as for almost contact metric manifold there is introduced notion of $\mathcal D$-homothety of 
almost paracontact metric manifold. 

Let $\mathcal M$ be $(2n+1)$-dimensional almost paracontact metric manifold. 
A local frame of vector fields $\xi$, $E_i$, $E_{i+n}$, $i=1,\ldots n$, is called Artin frame if
\begin{align}
& \phi E_i = E_i, \quad \phi E_{i+n} = - E_{i+n}, \quad i=1,\ldots n \\
& g(\xi, E_i)=g(\xi, E_{i+n})=0, \quad g(E_i, E_{i+n}) =1,\quad i=1,\ldots n.
\end{align}
Gauge of local Artin frame is deformation of Artin into another Artin frame defined by
\begin{align}
E_i\mapsto E'_i = f_i E_i, \quad E_i\mapsto E'_{i+n}=f^{-1}_iE_{i+n}, \quad i=1,\ldots n,
\end{align}
where $f_1,\ldots f_n$ is a family of locally defined function, non-zero everywhere on domain of their definition.

\subsection{Three dimensional connected simply connected Lie groups}
Curvature of arbitrary left-invariant Riemannian metric on 3-dimensional Lie group 
was described in simple and intuitive way by John Milnor in his 
paper \cite{Milnor}. In particular for unimodular groups there is  
\begin{theorem}[J.~Milnor \cite{Milnor}]
Let $\mathcal G$ be 3-dimensional unimodular Lie group,  
equipped with left-invariant Riemannian metric.  There is orthonormal frame $(e_1,e_2, e_3)$ of 
left-invariant vector fields and constants $\lambda_1$, $\lambda_2$, 
$\lambda_3$, such that 
\begin{equation}
[e_2,e_3]=\lambda_1 e_1, \quad 
[e_3,e_1] = \lambda_2 e_2, \quad
[e_1,e_2]= \lambda_3 e_3.
\end{equation}
Signs of $\lambda_i$
\footnote{The base change $e_i \mapsto -e_i$ follows $\lambda_i \mapsto -\lambda_i$.}
    up to the order determine $\mathcal G$ uniquely 
if $\mathcal G$ is connected and simply connected. Define 
$\mu_1$, $\mu_2$, $\mu_3$
\begin{equation}
\mu_i = \frac{1}{2}(\lambda_1+\lambda_2+\lambda_3)-\lambda_i.
\end{equation}
The orthonormal base $(e_1,e_2,e_3)$ diagonalizes Ricci quadratic form, 
the principal Ricci curvatures $r_i=r(e_i)$, $i=1,2,3$, being given by 
\begin{equation}
r_1=2\mu_2\mu_3, \quad
r_2= 2\mu_1\mu_3, \quad
r_3= 2\mu_1\mu_2.
\end{equation}
Let $v\times w$ be vector product determined by
$e_1\times e_2 = e_3$, $e_2\times e_3 = e_1$, $e_3\times e_1=e_2$. 
On Lie algebra connection maps are given by 
$x \mapsto \nabla_{e_i}x= \mu_i (e_i \times x)$. 
 \end{theorem}

\section{Three-dimensional contact metric and almost cosymplectic $(\kappa,\mu)$-manifolds}
Let $(\phi,\xi,\eta,g)$ be three-dimensional left-invariant almost paracontact metric structure 
on unimodular Lie group $\mathcal G$. We assume there is an orthonormal frame of 
left-invariant vector fields $(\xi,E_1,E_2)$, such that $\phi E_1 = E_2$ and
 commutators are given by
\begin{equation}
 [E_1,E_2]=2k \xi,\quad [E_2,\xi] = -(\lambda+c)E_1, \quad
 [\xi, E_1] = (\lambda-c)E_2,
\end{equation}
$k$, $\lambda$, $c$ are some constants.

 We define $2\times 2$-matrix $A=ad_\xi|_{\{E_1,E_2\}}$. 

\begin{proposition}
\label{kmA1}
Let matrix $A$ be non-nilpotent. The almost contact metric manifold $\mathcal M$ 
is $(k^2-\lambda^2, 2(k+c))$-space, ie.
\begin{equation}
\label{klc:curv}
R_{XY}\xi = (k^2-\lambda^2)(\eta(Y)X-\eta(X)Y) +
2(k+c)(\eta(Y)hX-\eta(X)hY).
\end{equation}
\end{proposition}
\begin{proof}
We employ Milnor's method \footnote{Mentioned earlier local isometry with Lie group justifies use of Milnor's theorem.} 
to find Levi-Cvita connection coefficients and 
then directly compute:
\begin{eqnarray}
 & & \label{rexx} R_{E_1\xi}\xi = (k^2-\lambda^2+2\lambda(k+c))E_1, \quad
R_{E_2\xi}\xi = (k^2-\lambda^2-2\lambda(k+c))E_2,  \\ [+4pt] 
& & \label{reex} R_{E_1 E_2}\xi = 0,   
\end{eqnarray}
from (\ref{reex})  it follows that 
\begin{equation}
R_{XY}\xi = \eta(Y)R_{X\xi}\xi-\eta(X)R_{Y\xi}\xi,
\end{equation}
Jacobi operator $J_\xi X= R_{X\xi}\xi$ by the first set of identities has 
decomposition
\begin{equation}
J_\xi X = (k^2-\lambda^2)(Id-\eta(X)\xi) + 2(k+c)h,
\end{equation}
therefore
\begin{eqnarray*}
R_{XY}\xi & = &  \eta(Y)R_{X\xi}\xi-\eta(X)R_{Y\xi}\xi  
                 =  (k^2-\lambda^2)(\eta(Y)X-\eta(X)Y) + \\
             & &    2(k+c)(\eta(Y)hX-\eta(X)hY).
\end{eqnarray*}
\end{proof}

Following Milnor's paper  we obtain 
\begin{proposition}
Pincipal Ricci curvatures of $\mathcal M$  are given by
\begin{equation}
\label{ricci}
\begin{array}{l}
r_1=Ric(\xi,\xi) = 2(k^2-\lambda^2), \quad 
r_2= Ric(E_1,E_1) = -2(k+c)(k-\lambda),\\ [+4pt]
r_3 = Ric(E_2,E_2) = -2(k+c)(k+\lambda),
\end{array}
\end{equation}
scalar curvature
\begin{equation}
s = \sum r_i =-2(k^2+\lambda^2)-4kc. 
\end{equation}
\end{proposition}

\begin{remark}[Classification - contact metric case, $|I_{\mathcal M}| \neq 1$ ] 
Let $k=1$. Then $\mathcal M$ is contact metric $(\kappa,\mu)$-manifold 
with $\kappa = 1-\lambda^2$, $\mu = 2(1+c)$. Boeckx invariant $|I_{\mathcal M}| \neq 1$.  
\footnote{ $|I_{\mathcal M}| =1$  implies $\lambda^2-c^2 = 0$ which contradicts assumptions of the Proposition \ref{mil}.}
Resolving these equations
$\lambda =\sqrt{1-\kappa}$, $\frac{\mu}{2}=c +1$, we obtain explicit form
\begin{equation}
\begin{array}{l}
[E_1,E_2]=2\xi,
\quad
[E_2,\xi]=-(\sqrt{1-\kappa}+\mu/2 -1)E_1, \\ [+4pt]
[\xi,E_1]=(\sqrt{1-\kappa}-\mu/2 +1)E_2, ,
 \end{array}
 \end{equation}
 In terms of Boeckx invariant 
for non-Sasakian manifolds  resp. coefficients can be expressed as 
\begin{equation}
(I_{\mathcal M}-1)\sqrt{1-\kappa}, \quad
(I_{\mathcal M}+1)\sqrt{1-\kappa}.
\end{equation}
  By \ref{ricci}, for non-Sasakian manifold we have 
 \begin{equation}
\label{rcm}
\begin{array}{l}
r_1=Ric(\xi,\xi) = 2\kappa, \quad 
r_2= Ric(E_1,E_1) = -\mu(1-\sqrt{1-\kappa}),\\ [+4pt]
r_3 = Ric(E_2,E_2) =- \mu(1+\sqrt{1-\kappa}),
\end{array}
\end{equation}
scalar curvature
\begin{equation}
s = r_1+r_2+r_3 = 2(\kappa-\mu).
\end{equation}
For example $Ric > 0$ if and only if $0 <\kappa < 1$, $\mu < 0 $. In this case manifold 
is compact with compact covering. By list below universal cover is 
$\mathbb S^3$. 
\end{remark}

\begin{remark}[Classification - almost cosymplectic case, $|C_{\mathcal M}| \neq 1$] 
Let $k=0$. The manifold $\mathcal M$ is almost cosymplectic with $|C_{\mathcal M}| \neq 1$. 
We find
$\lambda =\sqrt{-\kappa}$, $\mu = 2c$, explicit form
\begin{equation}
[E_1, E_2] =0,\quad [E_2,\xi]= -(\sqrt{-\kappa}+\mu/2)E_1, \quad
[\xi, E_1]=(\sqrt{-\kappa}-\mu/2)E_2, \quad
\end{equation}
 In terms of Dacko-Olszak invariant $C_{\mathcal M}$  
 corresponding coefficients  are given by
 \begin{equation}
(C_{\mathcal M} -1)\sqrt{-\kappa}, \quad 
(C_{\mathcal M}+1)\sqrt{-\kappa},
\end{equation}
 By (\ref{ricci}) principal curvatures 
 of Ricci tensor  being 
 given by
 \begin{equation}
 \begin{array}{l}
r_1=Ric(\xi,\xi)= 2\kappa, \quad 
r_2=Ric(E_1,E_1)= \mu\sqrt{-\kappa}, \\[+4pt]
r_3=Ric(E_2,E_2) = -\mu\sqrt{-\kappa},
\end{array}
\end{equation}
scalar curvature
\begin{equation}
s = r_1+r_2+r_3 = 2\kappa.
\end{equation}
For example we easily see that there are only two possibilities for signature 
of Ricci tensor in case manifold is non-cosymplectic: $(-1,0,0)$ or 
$(-1,-1,+1)$. 
\end{remark}

Case $A$ nilpotent is more complex, as there is no symmetry between conditions $\lambda-c=0$ and 
$\lambda +c = 0$. 
\begin{proposition} Let $A$ be nilpotent. If $\lambda -c =0$ the commutators of $\xi$, $E_1$, $E_2$ are given 
by
\begin{equation}
\label{nilp1}
[E_1,E_2] = 2k\xi, \quad [E_2,\xi] = - 2\lambda E_1, \quad [\xi, E_1] = 0.
\end{equation}
If $\lambda +c = 0$ they are given by
\begin{equation}
\label{nilp2}
[E_1, E_2] = 2k\xi, \quad [E_2, \xi] = 0, \quad [\xi, E_1] = 2\lambda E_2.
\end{equation}
\end{proposition}

\begin{proposition}
\label{kmA2}
Let $A$ be nilpotent. In case $\lambda - c =0$, the  manifold $\mathcal M$ is 
almost contact metric $(k^2-\lambda^2, 2(k+\lambda))$-space. In case $\lambda+c=0$, 
$\mathcal M$ is $(k^2-\lambda^2,2(k-\lambda))$-space.
\end{proposition}
\begin{proof}
Proof goes on in the same way as the proof of {\bf Proposition \ref{kmA1}}.
\end{proof}

\begin{corollary}[Contact metric structure, $I_{\mathcal M} = -1$]
In the Proposition above, for $k=1$, $\lambda-c=0$, $\mathcal M$ is contact metric manifold with 
Boeckx invariant $I_{\mathcal M}=-1$. Lie algebra of vector fields $(\xi, E_1, E_2)$ is isomorphic to the 
Lie algebra of left-invariant vector fields on the Lie groups of rigid motions of hyperbolic plane
\footnote{Plane with indefinite flat pseudo-metric} $ E(1,1)$.    
\end{corollary}

\begin{corollary}[Contact metric structure, $I_{\mathcal M} = 1$] In the case $k=1$, $\lambda+c=0$,
$\mathcal M$ is contact metric manifold with Boeckx invariant $I_{\mathcal M}=1$. Lie algebra of 
vector fields $(\xi, E_1, E_2)$ is isomorphic to the 
Lie algebra of left-invariant vector fields on the Lie group of rigid motions of Euclidean plane $E(2)$.  
\end{corollary}

Almost cosymplectic case is less troublesome, setting $k=0$ in (\ref{nilp1}) or in (\ref{nilp2}) we obtain Lie algebra 
of 3-dimensional  Heisenberg Lie group. The respective almost cosymplectic manifolds are $(\kappa,\mu)$-manifolds 
with Dacko-Olszak invariant $C_{\mathcal M} = -1$ in the case (\ref{nilp2}) or $C_{\mathcal M}=1$ for (\ref{nilp2}).

\section{Three-dimensional almost paracontact metric $(\kappa,\mu)$-manifolds}
\subsection{Family of examples of almost paracontact metric $(\kappa,\mu)$-spaces}
We study particular class of left-invariant almost paracontact metric structures on 
Lie groups. Let $\mathcal G$ be connected 3-dimensional Lie group and $(\phi, \xi, \eta, g)$ 
left-invariant structure on $\mathcal G$. Let $(\xi, E_1, E_2)$ be an Artin frame of left-invariant 
vector fields on $\mathcal G$. Lie algebra of $\mathcal G$ is given by 
\begin{align}
\label{lie:al:def}
& [E_1,E_2] = p_1E_1 +p_2 E_2+2u \xi, \\
& [\xi, E_1] = a E_1+ b E_2, \quad [\xi, E_2] = c E_1 + d E_2.
\end{align} 
Set $A = (\begin{smallmatrix} a & c \\ b & d \end{smallmatrix})$, 
$p =(\begin{smallmatrix} p_1 \\ p_2 \end{smallmatrix})$, $I$ - identity $2\times 2$-matrix. 
Jacobi identity requires 
\begin{equation}
u(a+d) = 0, \quad (A - (a+d)I)p = 0
\end{equation}

In what will follow if not otherwise explicitly stated we assume $p_1=p_2=0$, $a+d=0$. 
\begin{proposition}
\label{con:coef}
 Levi-Civita connection coefficients are given by 
\begin{align}
& \nabla_\xi\xi = 0, \quad \nabla_\xi E_1= l_3 E_1, \quad \nabla_\xi E_2 = -l_3 E_2, \\
& \nabla_{E_1}\xi = -k_1E_1-k_2E_2, \quad \nabla_{E_1}E_1 = k_2 \xi ,\quad \nabla_{E_1}E_2= k_1\xi,  \\
&  \nabla_{E_2}\xi = -m_1 E_1-m_2 E_2, \quad \nabla_{E_2}E_1 = m_2 \xi,\quad \nabla_{E_2}E_2= m_1\xi.  
\end{align}
\end{proposition}
\begin{proof}
For example $L(X,Y)=g(\nabla_\xi X,Y)$ is skew-symetric tensor, therefore matrix of coefficients 
is just common skew-symmetric $3\times 3$-matrix 
\begin{equation}
[L] =\begin{pmatrix}
0 & -l_1 & -l_2 \\
l_1 & 0 & -l_3 \\
l_2 & l_3 & 0
\end{pmatrix}.
\end{equation}
Notice that matrix of coefficients of $\nabla_\xi$, is obtained by switching two last rows in $[L]$. 
Due to constraints coming from Lie algebra structure there must be $l_1=l_2=0$. And we apply 
the same procedure to $\nabla_{E_i}$, $i=1,2$.
\end{proof}
As connection is torsion-less we have following relations between connection coefficients and 
Lie algebra structure constants
\begin{equation}
a=l_3+k_1= l_3-m_2, \quad b=k_2, \quad c=m_1, \quad k_1 = -m_2 \quad u=k_1=-m_2.
\end{equation}

Jacobi operator $J_\xi$ associated to the vector field $\xi$ is defined by 
$X \mapsto J_\xi X = R_{X\xi}\xi$, $R$ stands for curvature operator 
$R_{XY}Z = \nabla_X\nabla_YZ-\nabla_Y\nabla_X Z -\nabla_{[X,Y]}Z$.

\begin{proposition}
The curvature $R_{XY}\xi$, is determined 
by Jacobi operator $J_\xi$ 
\begin{equation}
R_{XY}\xi = \eta(Y)J_\xi X - \eta(X) J_\xi Y.
\end{equation}
\end{proposition}
\begin{proof}
It is enough to show that $R_{E_1E_2}\xi = 0$. So
\begin{equation}
\label{r:e1e2:xi}
\begin{array}{rcl}
R_{E_1E_2}\xi & = & \nabla_{E_1}\nabla_{E_2}\xi - \nabla_{E_2}\nabla_{E_1}\xi - \nabla_{[E_1,E_2]}\xi = \\
               & &   -m_1\nabla_{E_1}E_1-m_2 \nabla_{E_1}E_2 + k_1 \nabla_{E_2}E_1+k_2\nabla_{E_2}E_2 = \\
               & &        -(m_1k_2+m_2k_1)\xi +(m_1k_2+m_2k_1)\xi =0, 
\end{array}
\end{equation}
we have used $\nabla_{[E_1,E_2]}\xi = u\nabla_\xi\xi=0$.
\end{proof}

If Jacobi operator can be decomposed into $J_\xi = \kappa Id +\mu h$, then 
$\mathcal G$ equipped with such structure 
became almost  paracontact metric $(\kappa,\mu)$-space. 
Jacobi operator is symmetric in the sense that $g(J_\xi X,Y)=g(X, J_\xi Y)$. Possible 
non-zero coefficients are $a_{ij}=g(J_\xi E_i, E_j)$, $i,j=1,2$,  
$a_{ij}=a_{ji}$. Therefore
\begin{equation}
(g(J_\xi E_i,E_j)) = 
\begin{pmatrix} 
a_{11} & a_{21} \\
a_{12} & a_{22}
\end{pmatrix},
\end{equation}
$a_{12}=a_{21}= \kappa $ and
\begin{equation}
J_\xi E_1 =  \kappa E_1 +a_{11}E_2, \quad J_\xi E_2= a_{22}E_1+\kappa E_2,
\end{equation} 
 
For local components of $h=\dfrac{1}{2}\mathcal L_\xi \phi$, we have 
\begin{equation}
h\xi = 0, \quad h E_1 = k_2 E_2, \quad h E_2 = -m_1 E_1.
\end{equation}

Manifold is $(\kappa,\mu)$-manifold, iff
 $a_{11} = \mu k_2$ and $a_{22}= -\mu m_1$, for some constant $\mu$. We verify
\begin{eqnarray}
& a_{11} = g(J_\xi E_1, E_1) =R(E_1,\xi,\xi,E_1) = -2l_3 k_2, & \\
& a_{22} = g(J_\xi E_2, E_2) = R(E_2,\xi, \xi, E_2) = 2l_3 m_1. &
\end{eqnarray} 

So at this point we see that every manifold so far being considered is $(\kappa,\mu)$-manifold 
with 
\begin{align}
\label{apc:kappa}
& \kappa = a_{12} = R(E_1,\xi,\xi, E_2)= k_1m_2 - k_2m_1 = - ( u^2+bc ), \\
\label{apc:mu}
& \mu = -2l_3 = 2(u-a),
\end{align}

All considered examples  satisfy 
\begin{equation}
\label{deta:dphi}
d\eta = u \Phi,\quad d\Phi = 0.
\end{equation} 
For particular values $u=1$, $u=0$ we obtain paracontact metric $(\kappa,\mu)$-spaces, resp. 
 almost paracosymplectic $(\kappa,\mu)$-spaces.

Before proceeding further at first let focus on trying to find out $\mathcal D$-homothety invariants
similar to Boeckx or Dacko-Olszak invariants.

\begin{proposition}
Quantity $(u-\mu/2)^2/(u^2+\kappa)$, is $\mathcal D$-homothetically 
invariant 
\begin{equation}
\dfrac{(u-\mu/2)^2}{u^2+\kappa} = \dfrac{(u-\mu'/2)^2}{u^2+\kappa'}.
\end{equation}
\end{proposition}  

\begin{proof}
For Artin frame $(\xi, E_1, E_2)$ the frame $(\xi', E_1', E_2')$, 
$\xi' = \xi/\alpha$, $E_i' = E_i/\sqrt{\alpha}$, $i=1,2$ is Artin for structure deformed 
by $\mathcal D$-homothety. By (\ref{lie:al:def}), there is
\begin{equation}
[E_1',E_2']= 2u\xi',\quad [\xi', E_1']= a'E_1'+b'E_2', \quad [\xi',E_2']=c'E_1'-a'E_2'.
\end{equation}
Corresponding Lia algebras constants are related by
\begin{equation}
a' = \alpha a, \quad b' = \alpha b, \quad c'=\alpha c.
\end{equation}
 Moreover $\kappa$, $\kappa'$,  
$\mu$, $\mu'$ are determined in terms of these constants 
\begin{equation}
\begin{array}{lcllcl}
 \kappa & = & -u^2-bc, \quad & \kappa' & = & -u^2-b'c',  \\
 \mu & = & 2(u - a),  \quad & \mu' & = & 2(u -a'). 
\end{array} 
\end{equation}
Therefore by above relations $\mathcal D$-homothety coefficient 
\begin{equation}
\alpha^2 = \dfrac{u^2+\kappa}{u^2+\kappa'} = \dfrac{(2u-\mu)^2}{(2u-\mu')^2}.
\end{equation}
\end{proof}

For paracontact metric, $u=1$, $(\kappa,\mu)$-space, $\kappa\neq -1$ we set 
\begin{align}
\label{pBoeckx}
& E_{\mathcal M}= (1-\mu/2)^2/ (1+\kappa).
\end{align}
For almost  paracosymplectic, $u=0$, $(\kappa,\mu)$-space, $\kappa \neq 0$ we set 
\begin{align}
& F_{\mathcal M} = (-\mu/2)^2/\kappa.
\end{align}
By above remarks both $E_{\mathcal M}$, and $F_{\mathcal M}$ are $\mathcal D$-homothety 
invariant.
                    
Set $N^{(1)} = [\phi,\phi]-2d\eta\otimes \xi$, $N^{(2)}(X,Y)=\mathcal L_{\phi X}\eta(Y)-\mathcal L_{\phi Y}\eta(X)$,
\begin{proposition} For every almost paracontact metric manifold there is
\begin{align}
\label{cov:div:pfi}
& 2g((\nabla_X\phi)Y,Z)= -3d\Phi(X,\phi Y,\phi Z)-3d\Phi(X,Y,Z)-g(N^{(1)}(Y,Z),\phi X) + \\
 & \quad   N^{(2)}(Y,Z)\eta(X) +2d\eta(\phi Y,X)\eta(Z)-2d\eta(\phi Z,X)\eta(Y),  \nonumber
\end{align}
\end{proposition}
\begin{proof}
Proof is similar to the proof of analogous formula in case of almost contact metric manifolds. The latter can be found 
in \cite{Blair}
\end{proof}

\begin{corollary}
\label{cov:div:3}
 Three-dimensional almost paracontact metric manifold satisifies
\begin{align}
& g((\nabla_X\phi)Y,Z) = (f\Phi(Y,X) +d\eta(\phi Y,X)+g(hY,X))\eta(Z)- \\ 
& \qquad (f\Phi(Z,X)+d\eta(\phi Z,X)+g(hZ,X))\eta(Y), \nonumber \\
& d\Phi = 2f\eta\wedge\Phi 
\end{align}
\end{corollary}
\begin{proof}
On three-dimensional manifold 
\begin{equation*}
-3d\Phi(X,\phi Y,\phi Z)-3d\Phi(X,Y,Z) = 2f\eta(Y)\Phi(X,Z)-2f\eta(Z)\Phi(X,Y).
\end{equation*}
Moreover tensor fields $N^{(1)}$ and $N^{(2)}$ are zero on vector fields $Y$, $Z$, such that 
$\eta(Y)=\eta(Z)=0$. Having this in mind for vector fields $Y$, $Z$ we set $\bar Y=Y-\eta(Y)\xi$, 
$\bar Z=Z -\eta(Z)\xi$ and in virtue of  (\ref{cov:div:pfi}), we have  $g((\nabla_X\phi)\bar Y, \bar Z)=0$. Hence
\begin{equation*}
g((\nabla_X\phi)Y,Z) = \eta(Y)g(\nabla_X\xi, \phi Z)-\eta(Z)g(\nabla_X\xi,\phi Y),
\end{equation*}
and
\begin{equation*}
g(\nabla_X\xi, \phi Z) = \frac{f}{2}\Phi(X,Z)+d\eta(X,\phi Z)- g(X,hZ),
\end{equation*}
 $h=\frac{1}{2} \mathcal L_\xi\phi$. From these last two  identities we obtain our claim. 
\end{proof}

\subsection{Paracontact metric three-dimensional manifolds}
For paracontact metric manifold $f=0$, $d\eta =\Phi$. {\bf Corollary \ref{cov:div:3}} follows
\begin{align}
\label{cov:pfi}
& (\nabla_X\phi)Y = - g(X,Y-hY)\xi +\eta(Y)(X-hX), \\
\label{cov:xi} & \nabla_X\xi =-\phi X + \phi h X. 
\end{align}

\begin{proposition}
\label{prop1:cov:coef}
Let $\mathcal M$ be three-dimensional non para-Sasakian paracontact metric $(\kappa,\mu)$-space.
Let $p$ be a point of $\mathcal M$. Assume there is local vector field $X$, on neighborhood of $p$, that
$\phi X =X$, $h X \neq 0$. Then on neighborhood of $p$ there is a local Artin frame $\xi$, $E_1$, $E_2$, 
function $b$, $\epsilon =\pm 1$, that 
\begin{align}
& \nabla_\xi\xi = 0, \quad \nabla_\xi E_1= -\frac{\mu}{2} E_1, \quad \nabla_\xi E_2 = \frac{\mu}{2} E_2, \\
& \nabla_{E_1}\xi = -E_1-\epsilon E_2, \quad \nabla_{E_1}E_1 = bE_1+\epsilon \xi ,\quad 
\nabla_{E_1}E_2= -bE_2+\xi,  \\ %k_1=1
&  \nabla_{E_2}\xi = \epsilon(\kappa+1) E_1+ E_2, \quad \nabla_{E_2}E_1 = - \xi,\quad 
\nabla_{E_2}E_2= -\epsilon(\kappa+1)\xi,
\end{align}
if $\kappa \neq -1$ function $b$ vanishes.
\end{proposition}
\begin{proof}
Essentially proof is finished if we can show that  there is local Artin so coefficients of 
tensor field  $h$ are constants.  By assumption 
manifold is $(\kappa, \mu)$-space.  Therefore Jacobi operator $ X\mapsto J_\xi = R_{X\xi}\xi$, is given by 
$J_\xi = \kappa (Id -\eta\otimes\xi)+\mu h$. From other hand $J_\xi X = \nabla^2_{X,\xi}\xi -\nabla^2_{\xi,X}\xi$. 
With help of (\ref{cov:xi}) we find
\begin{align}
 \nabla^2_{X,\xi}-\nabla^2_{\xi,X}\xi = h^2 X -X+\eta(X)\xi -\phi (\nabla_\xi h)X.
\end{align}
So we obtain equation
\begin{align}
(\kappa+1)(Id-\eta\otimes\xi)+\mu h = h^2 -\phi (\nabla_\xi h).
\end{align}
Due to symmetries of $h$ the above equation splits into pair of separate equations
\begin{align}
& h^2 = (\kappa+1)(Id-\eta\otimes \xi), \\
\label{mu:h::cov:h} & \mu h = -\phi (\nabla_\xi h).
\end{align}
Note $h^2 = a_1a_2 (Id -\eta\otimes \xi)$, where $hE_1=a_1E_2$, 
$hE_2=a_2E_1$.
Hence $a_1a_2 = \kappa+1$. If necessary we gauge frame  to assure  $a_1=\epsilon=\pm 1$, then  
$a_2=\epsilon (\kappa+1)$. 

We are now going back to the proof. By $\nabla_\xi\phi=0$, (\ref{mu:h::cov:h}),  (\ref{cov:xi}), and (\ref{cov:pfi}) 
\begin{align}
& \nabla_\xi\xi =0, \quad \nabla_\xi E_1 =  f E_1, \quad \nabla_\xi E_2 = -fE_2, \\
& \mu E_2 =  \mu hE_1 = -\phi (\nabla_\xi h)E_1 = -2\epsilon f E_2,\\
& \nabla_{E_1}\xi = -\phi E_1+\phi hE_1 =-E_1-\epsilon E_2, \quad \nabla_{E_2}\xi = \epsilon(\kappa+1)E_1+E_2, \\
& \nabla_{E_1}E_1 = b E_1+\epsilon\xi, \quad \nabla_{E_1}E_2= -b E_2+\xi, \\
& \nabla_{E_2}E_1 = -cE_1-\xi, \quad \nabla_{E_2}E_2 = cE_2-\epsilon(\kappa+1)\xi,
\end{align}
with help of all these above formulas we obtain $R_{E_1E_2}\xi=\nabla_{E_1}\nabla_{E_2}\xi -
\nabla_{E_2}\nabla_{E_1}\xi - \nabla_{[E_1,E_2]}\xi = -2b(\kappa+1)E_1-2c E_2$. Therefore 
$c =0$ and $b=0$ if $\kappa \neq -1$.
\end{proof}

\begin{corollary}
\label{pcm:lie:comm}
Let $\mathcal M$ be three-dimensional non-para-Sasakian paracontact  metric $(\kappa,\mu)$-space. Then around 
each point there is a local Artin frame $\xi$, $E_1$, $E_2$, such that
\begin{align}
\label{e1:e2}
& [E_1,E_2]=-b E_2+2\xi, \\
\label{xi:e1:e2} 
& [\xi,E_1]= (1-\frac{ \mu}{2})E_1+\epsilon E_2, \quad [\xi, E_2]=-\epsilon(\kappa+1)E_1-(1-\frac{\mu}{2})E_2,
\end{align}
$b=0$ if $\kappa \neq -1$.
\end{corollary}

Conversely having vector fields as in the above corollary define almost paracontact metric structure 
taking $\xi$, $E_1$, $E_2$ as Artin frame. Such defined structure is almost paracosymplectic and directly 
we verify that manifold equipped with this structure is paracontact metric  $(\kappa,\mu)$-space, 
where the case $b\neq 0$ corresponds to $(-1,\mu)$-spaces.

By above corollary every three-dimensional paracontact metric $(\kappa,\mu)$-space, $\kappa\neq -1$, is 
locally isometric as almost paracontact metric manifold  to some three-dimensional connected 
simply connected Lie group equipped with left-invariant almost paracontact metric structure. 
For full list of such groups cf. Table \ref{tab:pcm}. 

\begin{example}[Family of non-isometric paracontact metric $(-1,\mu)$-spaces]
Under the assumptions of the {\bf Proposition \ref{prop1:cov:coef}} sectional 
curvature $K_{\mathcal D}$ of contact distribution of $(-1,\mu)$-space is given by 
\begin{align}
& K_{\mathcal D}: \mathcal D \rightarrow g(R_{E_2E_1}E_1,E_2) = E_2 b-(1+\mu).
\end{align}
Let function $b$ be  that $E_2 b= c$, $c\in \mathbb R$. Solutions 
leading to different values $c_1$, $c_2$ determine non-isometric three-dimensional 
paracontact $(-1,\mu)$-space $\mathcal M_1$, $\mathcal M_2$ as 
$K_{\mathcal D_1} = const. $, $K_{\mathcal D_2}=const.$, and 
$K_{\mathcal D_1} \neq K_{\mathcal D_2}$. Only what remains is to show that 
such solutions exist. By {\bf Corollary \ref{pcm:lie:comm}}, around every 
point there is local coordinate system $(t,x,y)$, that 
\begin{align}
& \xi = \partial_t, \quad E_2 = e^{-(1-\frac{\mu}{2})t}\partial_x, 
\end{align}
we set $b=b(t,x,y)= (K_{\mathcal D}+(1+\mu))e^{(1-\frac{\mu}{2})t}x$.
\end{example}

\subsection{Paracosymplectic $(\kappa,\mu)$-spaces}
For paracosymplectic manifolds
\begin{align}
\label{pcos:cov:pfi}
& (\nabla_X\phi)Y = g(X,hY)\xi-\eta(Y)hX, \\
\label{pcos:cov:xi}
& \nabla_X\xi = \phi hX.
\end{align}

\begin{proposition}
Let $\mathcal M$ be three-dimensional almost paracosymplectic non paraco\-symplectic $(\kappa,\mu)$-space.
Let $p$ be a point of $\mathcal M$. Assume there is local vector field $X$, on neighborhood of $p$, that
$\phi X =X$, $h X \neq 0$. Then on neighborhood of $p$ there is a local Artin frame $\xi$, $E_1$, $E_2$, 
function $b$, $\epsilon =\pm 1$, that
\begin{align}
& \nabla_\xi\xi = 0, \quad \nabla_\xi E_1= -\frac{\mu}{2} E_1, \quad \nabla_\xi E_2 = \frac{\mu}{2} E_2, \\
& \nabla_{E_1}\xi = -\epsilon E_2, \quad \nabla_{E_1}E_1 =  b E_1+\epsilon \xi ,
\quad \nabla_{E_1}E_2= -b E_2,  \\ %k_1=1
&  \nabla_{E_2}\xi =\epsilon \kappa E_1, \quad \nabla_{E_2}E_1 = 0,\quad \nabla_{E_2}E_2= -\epsilon\kappa\xi,
\end{align} 
if $\kappa \neq 0$, function $b$ vanishes $b=0$.
\begin{proof}
Proof goes on the same way as proof of {\bf Proposition \ref{prop1:cov:coef}}. The first we  
obtain for almost paracosymplectic $(\kappa,\mu)$-space
\begin{align}
& h^2 = \kappa (Id -\eta\otimes\xi), \\
& \mu h = -\phi(\nabla_\xi h),
\end{align}
so there is Artin frame $hE_1 = \epsilon E_2$, $hE_2 = \epsilon\kappa E_1$, $\epsilon =\pm 1$. 
Now we use (\ref{pcos:cov:pfi}), (\ref{pcos:cov:xi}) 
to obtain connection coefficients, finally we apply integrability condition $R_{E_1E_2}\xi =0$. 
\end{proof}
\end{proposition}

\begin{corollary}
\label{lie:comm2}
Let $\mathcal M$ be three-dimensional almost paracosymplectic $(\kappa,\mu)$-space. Then around 
each point there is a local Artin frame $\xi$, $E_1$, $E_2$, such that
\begin{align}
\label{pcos:e1:e2}
& [E_1,E_2]=-b E_2, \\
\label{pcos:xi:e1:e2} 
& [\xi,E_1]= -\frac{\mu}{2}E_1+\epsilon E_2, \quad [\xi, E_2]=-\epsilon\kappa E_1+\frac{\mu}{2}E_2,
\end{align}
if $\kappa \neq 0$, $b=0$.
\end{corollary}

Conversely having vector fields as in the above corollary define almost paracontact metric structure 
taking $\xi$, $E_1$, $E_2$ as Artin frame. Such defined structure is almost paracosymplectic and directly 
we verify that manifold equipped with this structure is almost paracosymplectic $(\kappa,\mu)$-space, 
where the case $b\neq 0$ corresponds to $(0,\mu)$-spaces.

Corollary above allow us to create full list of three-dimensional Lie groups admitting structure of paracosymplectic 
$(\kappa,\mu)$-spaces, $\kappa \neq 0$, up to $\mathcal D$-homothety, cf. Table \ref{tab:pcos}.

\begin{example}[Family of non-isometric $(0,\mu)$-spaces] 
For three-dimensional almost paracosymplectic $(0,\mu)$-space sectional curvature $K_{\mathcal D}$ is given by
\begin{align}
K_{\mathcal D}: \mathcal D \rightarrow g(R_{E_2 E_1}E_1,E_2) = E_2b.
\end{align}
If $\mathcal M_1$, $\mathcal M_2$, 
are two almost paracosymplectic $(0,\mu)$-spaces and $K_{\mathcal D_1} = const.$, $K_{\mathcal D_2}= const.$, 
$K_{\mathcal D_1}\neq K_{\mathcal D_2}$, such spaces are non-isometric as almost 
paracontact metric manifolds. For every isometry preserving Reeb fields must preserve sectional 
curvature of almost contact distributions. For every constant $c\in \mathbb R$, there is manifold 
 with $K_{\mathcal D}=c$. We just need to solve equations
\begin{align}
\xi b = -\frac{\mu}{2}, \quad E_2b =c, 
\end{align}
by {\bf Corollary \ref{lie:comm2}}, there is local coordinate system, such that $\xi =\partial_t$, 
$E_2 = e^{\frac{\mu}{2}t}\partial_x$, so one of possible solution is $b= K_{\mathcal D}\,e^{-\frac{\mu}{2}t}x$.
\end{example}

\newpage

 \begin{table}[h]
\begin{tabular}{|c c c| }
\hline
& & \\[-6pt]
 \small{ Lie group}
  & \small{Description} & 
  \small{ $I_{\mathcal M}=\frac{1-\mu/2}{\sqrt{1-\kappa}}$} \\%[+2pt]
  \hline
 & &  \\[-6pt]
 \Small{$ \mathcal {SO}(3)$ or $\mathcal {SU}(2)$} & 
 \Small{simple, compact, $\mathbb S^3$ or 
 $\mathbb S^{3}/\{\pm 1\}$} & \Small{$I_{\mathcal M} > 1$ } \\%[+2pt]
 \Small{$\mathcal {SL}(2,\mathbb R)$ or $\mathcal O(1,2)$} 
  & 
 \Small{simple, $\mathbb R^3$, compact quotients, 
 \cite{AusGrHa}}  &
 \Small{$I_{\mathcal M} < 1$, $I_{\mathcal M} \neq -1$} \\%[+2pt]
\Small{$\mathcal E(2)$} & 
 \Small{solvable, $\mathbb R^3$, compact quotients}  & 
 \Small{ $I_{\mathcal M} = 1$ } \\%[+2pt]
 \Small{$\mathcal E(1,1)$}   & 
 \Small{solvable, $\mathbb R^3$, compact quotients, 
 \cite{AusGrHa} } &
 \Small{$I_{\mathcal M} = -1$} \\%[+2pt]
  \hline
\end{tabular}
\vspace{.1cm}
\caption{\Small{Three-dimensional Lie groups with left-invariant, non-Sasakian 
contact metric $(\kappa,\mu)$-structures.}}
\label{tab:cm}
\end{table}

 \begin{table}[h]
\begin{tabular}{|c c c| }
\hline
& & \\[-6pt]
  \small{ Lie group}
  & \small{Description} & 
  \small{ $E_{\mathcal M}=\frac{(1-\mu/2)^2}{1+\kappa}$} \\%[+2pt]
  \hline
 & &  \\[-6pt]
 \Small{$ \mathcal {SO}(3)$ or $\mathcal {SU}(2)$} & 
 \Small{simple, compact, $\mathbb S^3$ or 
 $\mathbb S^{3}/\{\pm 1\}$} & \Small{$0< E_{\mathcal M} < 1$}\\%[+2pt]
  \Small{-} & 
 \Small{-} & \Small{$E_{\mathcal M}=0$ and $\kappa > -1$ } \\%[+2pt]
  \Small{$\mathcal {SL}(2,\mathbb R)$ or $\mathcal O(1,2)$} 
  & 
 \Small{simple, $\mathbb R^3$, compact quotients, 
 }  &
 \Small{$E_{\mathcal M} < 0$, $E_{\mathcal M} > 1$} \\%[+2pt]
\Small{-} & \Small{-} & \Small{$E_{\mathcal M}=0$, $\kappa < -1$} \\ %[+2pt]
 \Small{$E(2)$ or $E(1,1)$} &
\Small{solvable, $\mathbb R^3$, compact quotients,} &
\Small{ $E_{\mathcal M}=1$} \\%[+2pt]
  \hline
\end{tabular}
\vspace{.1cm}
\caption{\Small{Three-dimensional Lie groups with left-invariant, non para-Sasakian 
paracontact metric $(\kappa,\mu)$-structures, $\kappa\neq -1$.}}
\label{tab:pcm}
\end{table}

\begin{table}[h]
\begin{tabular}{|c c c| }
\hline
& & \\[-6pt]
 \small{ Lie group}
  & \small{Description} & 
  \small{$C_{\mathcal M}=\frac{-\mu/2}{\sqrt{-\kappa}}$} \\%[+2pt]
  \hline
 & &  \\[-6pt]
 \Small{$\mathcal E(2)$} & 
 \Small{solvable, $\mathbb R^3$, compact quotients}  & 
 \Small{ $|C_{\mathcal M}| > 1$ } \\%[+2pt]
 \Small{$\mathcal E(1,1)$}   & 
 \Small{solvable, $\mathbb R^3$, compact quotients } &
 \Small{$|C_{\mathcal M}| <1$} \\%[+2pt]
\Small{Heisenberg Lie group $\mathbb H^3$} & 
 \Small{nilpotent, $\mathbb R^3$, compact quotients}  & 
 \Small{ $|C_{\mathcal M}| = 1$ } \\%[+2pt]
  \hline
\end{tabular} 
\vspace{.1cm}
\caption{\Small{Three-dimensional Lie groups with left-invariant, non-cosymplectic 
 almost cosymplectic $(\kappa,\mu)$-structures, $\kappa < 0$.}}
\label{tab:acos}
\end{table}

\begin{table*}[h]
%\label{tab:pcos}
\begin{tabular}{|c c c| }
\hline
& & \\[-6pt]
 \small{ Lie group}
  & \small{Description} & 
  \small{$F_{\mathcal M}=\frac{(-\mu/2)^2}{\kappa}$} \\%[+2pt]
  \hline
 & &  \\[-6pt]
 \Small{$\mathcal E(2)$} & 
 \Small{solvable, $\mathbb R^3$, compact quotients}  & 
 \Small{ $0< F_{\mathcal M} < 1$ } \\%[+2pt]
\Small{-} & 
 \Small{-}  & 
 \Small{ $F_{\mathcal M} =0$, $\kappa>0$ } \\%[+2pt]
 \Small{$\mathcal E(1,1)$}   & 
 \Small{solvable, $\mathbb R^3$, compact quotients } &
 \Small{$|F_{\mathcal M}| >1$} \\%[+2pt]
\Small{-}   & 
 \Small{- } &
 \Small{$F_{\mathcal M} =0$, $\kappa < 0$} \\%[+2pt]
\Small{Heisenberg Lie group $\mathbb H^3$} & 
 \Small{nilpotent, $\mathbb R^3$, compact quotients}  & 
 \Small{ $F_{\mathcal M} = 1$ } \\%[+2pt]
  \hline
\end{tabular} 
\vspace{.1cm}
\caption{\Small{Three-dimensional Lie groups with left-invariant, non paracosymplectic 
 almost paracosymplectic $(\kappa,\mu)$-structures, $\kappa\neq 0$.}}
\label{tab:pcos}
\end{table*}

\end{document}